\documentclass[12pt, A4]{article}
\usepackage{amssymb}
\usepackage{amsmath,amsfonts,amsthm,amstext}
\usepackage[english]{babel}
\usepackage[symbol]{footmisc}

\DefineFNsymbols{stars}[text]{* {**} {*} {\S} {\I}}
\setfnsymbol{stars}

\newtheorem{thm}{Theorem}[section]
\newtheorem{cor}[thm]{Corollary}
\newtheorem{pro}[thm]{Proposition}
\newtheorem{deff}[thm]{Definition}
\newtheorem{lem}[thm]{Lemma}
\newtheorem{rem}[thm]{Remark}

\newcommand{\nc}{\newcommand}

\nc{\cc}{\D{C}} \nc{\hh}{\D{H}} \nc{\nn}{\D{N}} \nc{\oo}{\D{O}}
\nc{\qq}{\D{Q}}
 \nc{\rr}{\D{R}}
\nc{\zz}{\D{Z}} \nc{\livre}{\ast}

\nc{\barr}{\begin{array}} \nc{\earr}{\end{array}}
\nc{\bthm}{\begin{thm}} \nc{\ethm}{\end{thm}}
\nc{\bpro}{\begin{pro}} \nc{\epro}{\end{pro}}
\nc{\blem}{\begin{lem}} \nc{\elem}{\end{lem}}
\nc{\bins}{\begin{ins}} \nc{\eins}{\end{ins}}
\nc{\bcor}{\begin{cor}} \nc{\ecor}{\end{cor}}
\nc{\brem}{\begin{rem}} \nc{\erem}{\end{rem}}
\nc{\bdeff}{\begin{deff}} \nc{\edeff}{\end{deff}}
\nc{\bea}{\begin{eqnarray}} \nc{\eea}{\end{eqnarray}}
\nc{\D}[1]{{\mathbb#1}}
\def\R{\rm I\kern -.2em R}
\def\N{\rm I\kern -.18em N}
\def\Z{\rm Z\kern -.332em Z}
\def\de{\rm [\kern -.15em [}
\def\dd{\rm ]\kern -.15em ]}
\def\||{\hspace{0.15cm}|\hspace{0.15cm}}

\title{Bianchi groups are conjugacy separable}
\author{S. C. Chagas,
 \,\,\, P. A. Zalesskii
 \footnote {Both authors were supported by
CNPq.}}

\begin{document}
\maketitle

\begin{abstract}
We prove that non-uniform arithmetic lattices of
$SL_2(\mathbb{C})$ and consequently the Bianchi groups are
conjugacy separable. The proof is based on recent deep results of
Agol, Long, Reid and Minasyan. The conjugacy separability of
groups commensurable with Limit groups is also established.

\end{abstract}

\section{Introduction}
 The Bianchi groups are defined as $PSL_2(O_d)$, where
$O_d$ denotes the ring of integers of the field
$\mathbb{Q}(\sqrt{-d})$ for each square-free positive integer $d$.
These groups are classical objects investigated  first time in
1872 by Luigi Bianchi. The Bianchi groups have long been of
interest, not only because of their intrinsic interest as abstract
groups, but also because they arise naturally in number theory and
geometry.  They are   discrete subgroups of
$PSL_2(\mathbb{C})\cong Isom^{+}(\mathbb{H}_3)$, and the quotient
$\mathbb{H}_3$ modulo $PSL_2(O_d)$ is a finite volume hyperbolic
$3$-orbifold. We refer to \cite{EGM} and \cite{Fine} for further
information about Bianchi groups.

The recent big advance in the study of Bianchi groups was the
proof that they are subgroup separable (see Theorem 3.4 in
\cite{lr}). Another important residual property is conjugacy
separability on which we concentrate in this paper.

A group $G$ is conjugacy separable if whenever $x$ and $y$ are
non-conjugate elements of $G$, there exists some finite
quotient of $G$ in which the images of $x$ and $y$ are
non-conjugate. The notion of the conjugacy separability owes its importance
to the fact, first pointed out by Mal'cev \cite{M}, that the conjugacy problem has
a positive solution in finitely presented conjugacy separable groups.

It follows from a recent work of Minasyan \cite{Ashot} combined with the work of Agol, Long and Reid \cite{alr}, \cite{lr}
that Bianchi groups are virtually conjugacy separable, i.e., contain a finite index subgroup that are conjugacy
separable\footnote{This important observation was made to us by Henry Wilton in private  communication.}.
This does not prove however the conjugacy separability of Bianchi groups, since (in contrast with subgroup
separability) conjugacy separability is not preserved by commensurability (see \cite{G, CZ2, MM}).
We say that two groups of $PSL_2(\mathbb{C})$ are commensurable if their intersection has finite index in both of them.
More generally, two groups are (abstractly) commensurable if they contain isomorphic subgroups of finite index. In the rest of the paper we shall always use commensurability in the sense of abstract commensurability (after all two notions are equivalent since abstractly commensurable groups are commensurable in some group).

In this paper we establish that (abstract) commensurability invariance of conjugacy separability
within torsion free groups very much depends on the centralizers of elements (see Theorem \ref{subida para extensao com cond central.}).
Using it and different methods for torsion elements
 we prove our main
\medskip
\begin{thm}\label{arithmetic manifolds}
Non-uniform arithmetic lattices of $SL_2(\mathbb{C})$ are conjugacy separable.
\end{thm}
\medskip
The conjugacy separability of Bianchi groups was conjectured in
\cite{WZ} where the conjugacy separability of Euclidean Bianchi
groups was proved (more precisely, the cases $d=1,2,7, 11$ were
established there and $d=3$ was completed in \cite{LZ}). Recall that non-uniform arithmetic lattices in $PSL_2(\mathbb{C})$ (or equivalently in $SL_2(\mathbb{C})$) are precisely the subgroups commensurable with Bianchi groups (see Theorem 8.2.3 \cite{MR-03}). Therefore Theorem \ref{arithmetic manifolds} imply conjugacy separability of all Bianchi groups.

\begin{thm} The Bianchi groups are conjugacy separable.\end{thm}

\bigskip

The same methods also  allow us to prove conjugacy separability of
virtually Limit groups.
\medskip
\begin{thm}
Virtually Limit groups are conjugacy separable.
\end{thm}

{\bf Acknowledgements.} The second author is very grateful to Henry Wilton and Ashot Minasyan for
useful conversations.

\section{ General Results}

The \emph{profinite topology} on a group $G$ is the topology where the
collection of all finite index normal subgroups of $G$ serves as a
fundamental system of neighborhoods of the identity element $1\in
G$, turning $G$ into a topological group.  The completion
$\widehat{G}$ of $G$ with respect to this topology is called the
profinite completion of $G$ and can be expressed as an inverse
limit
$$\widehat{G} =\lim\limits_{\displaystyle\longleftarrow\atop N}G/N, $$
 where $N$ ranges over all finite index normal subgroups of $G$.
 Thus $\widehat{G}$ becomes a profinite group, i.e. a compact
 totally disconnected topological group. Moreover, there exists a
 natural homomorphism $\iota: G \longrightarrow \widehat{G}$  that
 sends $g \mapsto(gN)$, this homomorphism  is a monomorphism when $G$ is
residually finite. If $S$ is a subset of a topological group
$\widehat{G}$, we denote by $\overline{S}$ its closure in
$\widehat{G}$. The profinite topology on $G$ is induced by the
topology of $\widehat{G}$. Note that for a subgroup $H$ of $G$,
the profinite topology  of $H$ can be stronger than the topology
induced by the profinite topology of $G$;  these topologies coincide iff
$\overline H=\widehat H$, for example this is the case  if $G$ is finitely
generated and $H$ is of finite index.

We say that $g\in G$ is \emph{conjugacy distinguished} if its conjugacy
class $g^G$ is closed in the profinite topology of $G$. For residually finite
$G$ this exactly means $g^{\widehat G}\cap G=g^G$.  Note that $G$ is
conjugacy separable iff every element of $G$ is conjugacy
distinguished.

\begin{pro}\label{conjugacy distinguished}
Let $G$ be a  finitely generated group containing a conjugacy separable normal subgroup $H$
of finite index. Let $a\in G$ be an element  such that there
exists a natural number $m$ with $a^m\in H$ and the following conditions
hold:
\begin{itemize}
\item $C_G(a^m)$  is conjugacy separable;
\item $\widehat{C_H(a^m)}= \overline{C_H(a^m)}= C_{\widehat{H}}(a^m)$.
\end{itemize}
Then $a$ is conjugacy distinguished.
\end{pro}
\begin{proof}
Since $H$ is conjugacy separable it is residually finite and therefore so is $G$.
Suppose $b=\gamma^{-1} a \gamma$, for some $\gamma\in \widehat G$, it means
that  $b\in a^{\widehat{G}}\cap G$ and then we need to prove that $b= a^g$, for
some $g\in G$.
Observe that $\widehat{G} = G\widehat{H}$, so that we can write
$\gamma = \delta\gamma_0$, where $\gamma_0\in \widehat{H}$ and
$\delta\in G$. Therefore $b^m = (a^{\gamma})^m =(a^{m})^{\gamma} =
(a^m)^{\delta\gamma_0}$. Now substituting $a$ by $a^{\delta}$, we
can suppose that $\gamma\in \widehat{H}$. Thus $a^m$ and $b^m$ are
conjugate in $\widehat{H}$, and since $H$ is conjugacy separable
there exists $h\in H$ such that $ a^m= (b^{h})^m$. Hence $ a^m =
(b^m)^{h}$, so $a^m$ and $b^m$ are conjugate in $H$. Thus we can
suppose that $a^m = b^m$ and so $\gamma\in C_{\widehat{H}}(a^m)$.

Let $C= C_G(a^m)$ be the centralizer of $a^m$ in $G$, so $b\in C$.
By the second hypothesis $\overline{C_H(a^m)}=
C_{\widehat{H}}(a^m)=\widehat{C_H(a^m)}$, so $\gamma\in
\widehat{C_G(a^m)}$. Now observe that $a, b\in C$ and $\gamma\in
\widehat{C_G(a^m)}$, so
 by the first hypothesis there exists $g\in G$ such that
$a^g= b$.
\end{proof}

\begin{deff} We say that a group $G$ is hereditarily conjugacy
separable if every finite index subgroup of $G$ is conjugacy
separable.\end{deff}

\begin{rem}\label{Minasyan} By Proposition 3.2 in \cite{Ashot} $H$ is hereditarily
 conjugacy separable iff $\overline{C_H(h)}=C_{\overline H}(h)$ for every $h\in H$,
 so hereditary conjugacy separability of $H$ would imply  the second equality of the second condition
of Proposition \ref{conjugacy distinguished}.
\end{rem}

Proposition \ref{conjugacy distinguished} implies the following

\begin{thm}\label{subida para extensao com cond central.}
Let $G$ be a finitely generated torsion free group containing a conjugacy separable
normal subgroup $H$ of finite index. Suppose that for every $1\neq h\in
H$,
\begin{itemize}
\item $C_G(h)$  is conjugacy separable;
\item $\overline{C_H(h)}= \widehat{C_H(h)}= C_{\widehat{H}}(h)$.
\end{itemize}
Then $G$ is conjugacy separable.
\end{thm}

\begin{thm}\label{order p} Let $G$ be a finitely generated subgroup separable group  having a finite index hereditarily
conjugacy separable  subgroup $H$ and a free normal subgroup $F$
of $H$ such that $H/F$ is   polycyclic. Then every element of prime
order $p$  in $G$ is conjugacy distinguished in $G$.
\end{thm}
\begin{proof}
Pick $a, b\in G$ such that $a = b^{\gamma}$, where $\gamma\in
\widehat{G}$ and $a, b$ have order $p$.
 We need to prove that $a$ and $b$ are conjugate in $G$.

Without loss of generality we can suppose that $H$ and $F$ are
normal in $G$. Indeed we can  replace $H$ and $F$ by their cores
(the core of a subgroup is the intersection of all its conjugates)
$H_G$ and  $F_G= \displaystyle\cap_{g\in G}F^{g}$; then since
$H/F_G$ is a subgroup of $\prod_{r} H/F^r$ and so is nilpotent,
where $r$ ranges over representatives of $G/H$, we can take the
core $(H/F_G)_{G/F_G}$ of $H/F_G$ in $G/F_G$ and replace $H$ by
the preimage of $(H/F_G)_{G/F_G}$ in $G$.

Since $\widehat{G}= G\widehat{H}$ we can write $\gamma= g \delta
$, where $\delta\in \widehat{H}$ and $g\in G$ and so $a=
b^{\gamma}= b^{g \delta}$. Therefore, changing $b$ to $b^{g}$ we
may assume that $\gamma\in \widehat H$ and so $b\in \langle
\widehat H,a\rangle\cap G=\langle H,a\rangle$. Thus we can suppose
that $G= \langle H, a\rangle$.

Put $G_0= G/F$. We shall use subindex $0$ for the images in $G_0$.
Since $G_0$ is polycyclic group it is conjugacy separable (cf.
\cite{Formanek}, \cite{Remelennikov} and \cite{Segal}), so we may
assume that $a_0=b_0$ and so $\gamma_0\in C_{\widehat{H_0}}(a_0)$.

Let $U$ be the preimage of the centralizer
$C_{G_0}(a_0)=C_{H_0}(a_0)\times \langle a_0\rangle$ and let $V$
be the preimage of $\langle a_0\rangle$ in $G$.  By Theorem of
Dyer and Scott \cite{DS} $V= \ast_{i\in I}(F_i \times T_i)\ast L$,
where $T_i = \langle t_i \rangle$ and $t_i$ has order $p$ and $F_i, L$ are free groups. Since
torsion elements are conjugate to elements of  free factors,
 we may assume that $a= t_k$ and $b= t_j$ for some $k,j \in I$. Moreover, since $H_0$ is
 polycyclic
  by Proposition 3.3 in \cite{R-S-Z} $\overline{C_{H_0}(a_0)}= C_{\widehat
  H_0}(a_0)$. Thus using equalities $\widehat G_0=\widehat G/\overline F$ and
  $\overline U/\overline F=\overline{C_{G_0}(a_0)}$  we deduce that $\gamma\in \overline U$. We shall prove that $a$
  and $b$ are conjugate in $U$.

 Consider the abelianization $V/ V'$ of $V$ and  observe that the torsion elements are conjugated in $V$ if and only
 if they coincide in $V/V'$. Therefore, $a$ and $b$ are conjugate in $U$ if, and only if, their images are conjugated in
 $U/V'$. Thus it is enough to prove that the images of $a$ and $b$ are conjugate in
 $U/V'$.

 Suppose not. We use bar  to denote the images of elements  in $U/V'$.
 Observe that $N:=U/V=C_{H_0}(a_0)$ acts on $V/V'$ and $\bar a$ and $\bar b$ are
conjugate in $U/V'$ iff $\bar a$ and $\bar b$ are in the same
$N$-orbit. Note also that since $U$ permutes torsion elements of
$V$, $N$-submodule $M$ of torsion elements in $V/V'$ is
permutational, i.e. is isomorphic to $\oplus_i
\mathbb{F}_p[N/N_i]$, where $N_i$ runs via subgroups of $N$
(namely $N_i=C_N(t_i)$ in our case) and the images of
non-conjugate torsion elements $t_i$ are canonical generators of
different summands (here $\mathbb{F}_p[N/N_i]$ means
$\mathbb{F}_p$-vector space with the basis $[N/N_i]$). Since $a$
and $b$ are not conjugate, $\bar a$ and $\bar b$ are in different
summands $M_a=\mathbb{F}_p [N/N_i]$ and $M_b=\mathbb{F}_p[
N/N_j]$. Put $M_{ab}=M_a\oplus M_b$.

Let $x_1,\ldots, x_n$ be elements of $U$  that whose images modulo
$V$ generate $N$ (recall that $N$ is finitely generated since it
is a subgroup of a finitely generated nilpotent group $H_0$).
Consider the group $R=\langle a,b, x_1,\ldots x_n\rangle$. Then
$RV/V=N$ and we have the following embedding
$M_{ab}\longrightarrow RV'/ V'$.  The elements  $\bar a$ and $\bar
b$ are not conjugate in $RV'/V'$ since they are in different
direct summands $M_a$ and $M_b$ of $M_{ab}$. Moreover, the images
of $\bar a$ and $\bar b$ in the quotient group $U_N$ of $RV'/V'$
modulo the normal subgroup $[N,M_{ab}]$ are central distinct and
so are non-conjugate.  Note that $U_N$ is finitely generated
abelian and so is residually finite. Since $M_{ab}/[N,M_{ab}]\cong
C_p\times C_p$ is finite, there exists a normal subgroup of finite
index $B$ in $U_N$ intersecting trivially $M_{ab}/[N,M_{ab}]$.
Thus the images of $a$ and $b$  in a finite quotient group $U_N/B$
of $R$ are distinct central elements.

On the other hand, since $G$ is subgroup separable $\overline
R=\widehat R$. It follows that there is a natural epimorphism
$\overline R\longrightarrow U_N/B$. Since the intersection
$\overline V'\cap \overline R$ of the closures of $V'$ and $R$ in
$\widehat G$ is in the kernel of the natural epimorphism
$\overline R\longrightarrow U_N/B$ we deduce that this epimorphism
factors via $\overline R\overline V'/\overline V'$. Thus to arrive
at contradiction it suffices to show that the images of $a$ and
$b$ in $\overline R\overline V'/\overline V'$ are conjugate.

Observe that $\overline N:=\overline U/\overline V$ acts on
$\overline V/\overline V'$. Since a finitely generated
abelian-by-polycyclic group is residually finite
(see Theorem 7.2.1 in \cite{LR}), $H/V'$ is residually
finite and so  $\overline V'\cap H=V'$. This means
that  $N$-module $V/V'$ embeds naturally in
$\bar N$-module $\overline V/\overline V'$. Thus the composition
of the natural homomorphisms $M\longrightarrow V/V'\longrightarrow
\overline V/\overline V'$ is injection and so  $M$ embeds
naturally in $\overline V/\overline V'$. Denote by $\overline
M_{ab}$ the closure of $M_{ab}$  in $\overline V/\overline V'$.
Since $\bar a$ and $\bar b$ are conjugate in $\overline
U/\overline V'$, $\bar a$ and $\bar b$ are in the same $\overline
N$-orbit. Since $\overline R\overline V/\overline V=\overline N$,
$\bar a$ and $\bar b$ are conjugate in $\overline R\overline
V'/\overline V'$ as required.
\end{proof}

\begin{cor}\label{free-by-polycyclic-by-p} Let $G=H\rtimes C_p$
be a semidirect product of a finitely generated
hereditarily conjugacy separable subgroup separable
free-by-polycyclic group $H$ and a group $C_p$ of prime order $p$.
Suppose  that for every $1\neq h\in H$ the centralizer
$C_G(h)$  is hereditarily conjugacy separable and
$\overline{C_H(h)}=\widehat{C_{H}(h)}$ in $H$. Then $G$ is hereditarily
conjugacy separable.\end{cor}

\begin{proof} Follows from Theorem \ref{order p} and Proposition
\ref{conjugacy distinguished} combined with Remark
\ref{Minasyan}.\end{proof}

\begin{thm}\label{free by polycyclic}
Let $G$ be a finitely generated subgroup separable group. Suppose
there exists a free-by-polycyclic hereditarily conjugacy separable
subgroup $H$ of $G$ of finite index. If $C_G(h)$ is hereditarily
conjugacy separable for every $h\in H$ and $C_G(g)$ is finitely
generated for every $g\in G$, then $G$ is conjugacy separable.
\end{thm}
\begin{proof} We need to prove that any element $a\in G$ is conjugacy distinguished. Suppose $a$
has finite order. Let $t$ be an integer such that $c=a^t$ has
prime order $p$.  Replacing $H$ by its core we may assume that $H$
is normal in $G$.  By Corollary \ref{free-by-polycyclic-by-p}
$G_1:=\langle c, H\rangle$ is hereditarily conjugacy separable and
so by Remark \ref{Minasyan} $ C_{{G_1}}(c)$ is dense in
$C_{\widehat G_1}(c)$. Since $G$ is subgroup separable and
$C_{G_1}(c)$ is finitely generated,
$\overline{C_{G_1}(c)}=C_{\widehat G_1}(c)$. Then by Proposition
2.1 (applied to the pair $(G,G_1)$) $a$ is conjugacy
distinguished.

  Suppose now $a$ has infinite
order.  Choose any conjugacy separable subgroup $H$ of finite
index in $G$.  Then $a^m\in H\setminus \{1\}$ for some $m\in
\mathbb{N}$, so similarly to the previous paragraph the result
follows from Proposition 2.1 combined with Remark 2.2.
\end{proof}

\section{Groups commensurable with Bianchi and Limit groups}

The proof of the next proposition was communicated to us by Henry
Wilton.

\begin{pro}\label{non-uniforme arithmetic}
A non-uniform arithmetic lattice in $SL_2(\mathbb{C})$ possesses a
(finitely generated free)-by-cyclic hereditarily conjugacy
separable subgroup of finite index.
\end{pro}
\begin{proof} Recall that non-uniform arithmetic lattices in $SL_2(\mathbb{C})$ are precisely
 the subgroups commensurable with Bianchi groups (see Theorem 8.2.3 \cite{MR-03}). Therefore
we just need to show the existence of a (finitely generated
free)-by-cyclic hereditarily conjugacy separable subgroup of
finite index in a Bianchi group $G$. By \cite{alr} Bianchi groups
are (virtually) geometrically finite subgroups of right angled
Coxeter groups obtained as  groups of isometries of the hyperbolic
space $\mathbb H^n$ (for some $n$) generated by the reflections
about co-dimension one faces of some right angled polyhedron.
Therefore by Theorem 1.6 in \cite{lr} Bianchi groups have a
subgroup $H$ of finite index that is virtual retract of a right
angled Coxeter group. By Corollary 2.3 in \cite{Ashot} a right
angled Coxeter group is virtually hereditarily conjugacy
separable. Since a virtual retract of a hereditarily conjugacy
separable group is hereditarily conjugacy separable (see Theorem
3.4 in \cite{CZ2}) it follows that $H$ is hereditarily conjugacy
separable. On the other hand by \cite{A} there exist a finite
index surface-by-infinite cyclic subgroup $U=S\rtimes \mathbb{Z}$
in the Bianchi group $G$. It is well known that Bianchi groups
have virtual cohomological dimension 2 ( see Theorem 11.4.4 \cite{BS-73}),
so from \cite[Corollary 1]{B2} $S$ is finitely generated free. Thus $U\cap H$ is the
desired subgroup.
\end{proof}

\begin{thm}\label{lattice}
Non-uniform arithmetic lattices of $SL_2(\mathbb{C})$ are
conjugacy separable.
\end{thm}
\begin{proof}
Let $G$ be a non-uniform arithmetic lattice of $SL_2(\mathbb{C})$.
By Proposition \ref{non-uniforme arithmetic} there exists a finite
index hereditarily conjugacy separable subgroup $H$ of $G$ such
that $H= F\rtimes \mathbb{Z}$ with $F$ free of finite rank. Then
for every $h\in H$ the centralizer $C_G(h)$ is finitely generated
virtually abelian by Lemma 2.2 in \cite{CZ3} and therefore is
hereditarily conjugacy separable. The equality $\overline{C_H(h)}=
\widehat{C_H(h)}$ holds since $H$ is subgroup separable (see
Theorem 3.4 in \cite{lr}).
 Thus the result follows
from Theorem \ref{free by polycyclic}.
\end{proof}

Since non-uniform arithmetic lattices in $PSL_2(\mathbb{C})$ (or
equivalently in $SL_2(\mathbb{C})$) are precisely the subgroups
commensurable with Bianchi groups (see Theorem 8.2.3 \cite{MR-03})
Theorem \ref{lattice} implies the conjugacy separability of the
Bianchi groups.

\begin{thm}
The Bianchi groups are congugacy separable.
\end{thm}

We come to limit groups. Recall that a fully residually free group
is a group that satisfies the following condition:
 for each finite subset $K\subset G$ of  elements, there exist a
homomorphism $\varphi : G \rightarrow F$ to some free group $F$
such that the restriction $\varphi_{|_K}$ is injective. A finitely
generated fully residually free group is called a limit group.

We recall now a construction of a limit group. Let $F$ be a free
group of finite rank and put ${\cal{Y}}_1= F$. For $i > 1$, define
the class ${\cal{Y}}_i$ to consist of all groups that are free
products $G_i=G_{i-1}\ast_C A$ of a group $G_{i-1}\in{\cal
Y}_{i-1}$ and a free abelian group $A$ of finite rank amalgamating
self-centralizing  cyclic subgroup of $G_{i-1}$ with a subgroup of
$A$ generated by a generater of $A$ (this construction is known as
an extension of the centralizer). If $L$ is a limit group then by
Theorem 4 in \cite{khamya2}, there exists $n$ such that $L$ embeds
in some $G_n\in {\cal Y}_n$.
\medskip
We shall need now the following abstract analogue of Theorem 1 in
\cite{KZ-07}.

\begin{thm}\label{cohomology} Let $H$ be a  group of finite cohomological dimension
$cd(H) = d$ such that  for every $i \geq 1$ the cohomology group
$H^i (H, \mathbb{F}_p)$ is finite and $\sigma$
 be an  automorphism of $H$ of order $p$. Then for the
 group $P$ of fixed points of $\sigma$ we have
 $$
 \sum_{j \geq 0} dim H^j (P, \mathbb{F}_p) \leq \sum_{j \geq 0} dim H^j (H, \mathbb{F}_p)
 .$$
 \end{thm}

 \begin{proof} The proof is a repetition of the proof of Theorem 1
 in \cite{KZ-07} with the use of  Remark after Theorem 7.4 in
 \cite{Brown2} instead of Theorem 2 in \cite{KZ-07}.\end{proof}

\begin{thm} A group commensurable with a limit group is conjugacy
separable.\end{thm}

\begin{proof} Let $G$ be a group commensurable with a limit group $L$. We
have to show that every element $g$ of $G$ is conjugacy
distinguished. Replacing $L$ by a common subgroup of finite index
in $G$ and $L$ we may assume that $L$ is a subgroup of finite
index of $G$. Then replacing $L$ by its core in $G$ we may assume
that $L$ is normal in $G$. We shall show that a pair $(G,L)$
satisfy hypothesis of Theorem \ref{free by polycyclic}.

Subgroup separability of limit groups \cite{W-2006} implies that
$G$ is subgroup separable. Since a subgroup of finite index of a
limit group is a limit group the main result of \cite{CZ1} shows
that $L$ is hereditarily conjugacy separable.
 By Theorem 2
in \cite{desi} there is a term of lower central series
$\gamma_i(L)$ which is free, so $L$ is finitely generated
free-by-nilpotent and in particular free-by-polycyclic.
 By commutative transitivity property
$C_L(h)$ is abelian for every $h\in L$ and it is well-known that
an abelian subgroup of a limit group is finitely generated
(follows from finiteness of cohomological dimension for example).
Therefore $C_G(h)$ is finitely generated virtually abelian for
every $h\in L$ and hence is hereditarily conjugacy separable. We
are left with checking that $C_G(g)$ is finitely generated for
every $g\in G$. Clearly it is equivalent to $C_L(g)$ being
finitely generated.

If $g$ has infinite order then some of its power is in $L$ and so
(as was explained above) $C_G(g)$ is finitely generated virtually
abelian. Suppose now $g$ is of finite order and $h:=g^m$ is of
prime order $p$ for some natural $m$.

 Since  limit groups are of finite
cohomological dimension and is of type $FP_\infty$ that in turn
implies finiteness of cohomology with coefficients in a finite
field we can apply Theorem \ref{cohomology} to deduce that  $dim
H^1(C_L(h),\mathbb{F}_p)$ is finite. Since
$H^1(C_L(h),\mathbb{F}_p)=Hom(C_L(h),\mathbb{F}_p)\cong
C_L(h)/([C_L(h),C_L(h)]C_L(h)^p)\cong H_1(C_L(h),\mathbb{F}_p)$,
it follows from Theorem 2 together with the last line of the first
paragraph of Section 4 in \cite{BH} that $C_L(h)$ is finitely
generated. Since $g$ normalizes $C_L(h)$ we can apply induction on
the order of $g$ to deduce that $C_{C_L(h)}(\langle
g\rangle/\langle h\rangle)=C_L(g)$ is finitely generated, as
required.

Thus the pair $(G,L)$ satisfies the premises of Theorem \ref{free
by polycyclic} according to which  $G$ is conjugacy separable. The
proof is complete.\end{proof}

\bigskip
\noindent Department of Mathematics\\
 University of Bras\'{\i}lia\\
Bras\'{\i}lia-DF 70910-900\\
 Brazil\\
 sheila@mat.unb.br, pz@mat.unb.br

\end{document}